\documentclass{article}

\usepackage{float,enumitem,amsthm,amsmath,amsfonts,amssymb,graphicx,latexsym,mathrsfs}

\usepackage{fullpage,setspace}

\usepackage{pgfplots}

\theoremstyle{plain}
\newtheorem{thm}{Theorem}[section]
\newtheorem{lem}[thm]{Lemma}

\newtheorem{cor}[thm]{Corollary}
\theoremstyle{definition}
\newtheorem{defn}{Definition}[section]

\newtheorem{exmp}{Example}[section]
\theoremstyle{remark}
\newtheorem{rem}{Remark}
\newtheorem{note}{Note}

\newcommand{\beeq}{\begin{equation}}
\newcommand{\eneq}{\end{equation}}
\newcommand{\bear}{\begin{eqnarray}}
\newcommand{\enar}{\end{eqnarray}}
\newcommand{\bearno}{\begin{eqnarray*}}
\newcommand{\enarno}{\end{eqnarray*}}

\newcommand{\expect}[1]{{\mathbb E}\left\{{\displaystyle #1}\right\}}
\newcommand{\expectbig}[1]{{\mathbb E}\Big\{{\displaystyle #1}\Big\}}
\newcommand{\expectsmall}[1]{{\mathbb E}\{{\displaystyle #1}\}}
\newcommand{\pr}[1]{{\mathbb P}\left\{{\displaystyle #1}\right\}}
\newcommand{\prm}[1]{{\mathbb P}^m\left\{{\displaystyle #1}\right\}}

\newcommand{\comps}{{H}}

\begin{document}

\title{Stability of Random Admissible-Set Scheduling \\
in Spatially Continuous Wireless Systems}

\author{N. Bouman\thanks{Eindhoven University of Technology, P.O. Box 513, 5600 MB Eindhoven, The Netherlands.} \and S.C. Borst\footnotemark[1] \thanks{Alcatel-Lucent Bell Labs, P.O. Box 636, Murray Hill, NJ 07974-0636, USA.} \and J.S.H. van Leeuwaarden\footnotemark[1]}

\date{}

\maketitle

\begin{abstract}
\noindent We examine the stability of wireless networks whose users are distributed over a compact space. A subset of users is called {\it admissible} when their simultaneous activity obeys the prevailing interference constraints and, in each time slot, an admissible subset of users is selected uniformly at random to transmit one packet. We show that, under a mild condition, this random admissible-set scheduling mechanism achieves maximum stability in a broad set of scenarios, and in particular in symmetric cases. The proof relies on a description of the system as a measure-valued process and the identification of a Lyapunov function.


\end{abstract}

\section{Introduction}
This paper examines the stability of a broad class of wireless networks whose users arrive to a compact space, according to some spatial stochastic process. Users independently take their locations at random according to some general distribution. Each user has a random, generally distributed, number of packets to transmit. 
Time is divided into time slots, and each user can transmit at most one packet per time slot.

A subset of users is called {\it admissible} when their simultaneous activity obeys the prevailing interference constraints. In each time slot, a subset of users is selected for transmission from all admissible sets uniformly at random and each of the users in the selected subset transmits one packet during the time slot. In practice, the relevant interference constraints depend on various system-specific properties, such as the propagation environment and the operation of the physical and medium-access layers of the network. In the present paper we therefore adopt generic feasibility criteria, which in particular cover both the SINR (Signal-to-Interference-and-Noise Ratio) model and the protocol model~\cite{GK00} as two canonical models for interference.

For wireless networks we are primarily interested in interference constraints that become looser when users are further apart and in one- and two-dimensional spaces. Our results, however, hold in more generality, which is why we present our results in terms of a particle system on a compact space of arbitrary dimension, to which batches of particles arrive according to some general stochastic process. We investigate under which conditions this particle system is stable, in the sense that the system is cleared from all particles infinitely often.

We thus investigate stability in the context of a model that combines a scheduling discipline operating under
interference constraints and a continuous spatial setting.
While these two elements have each been considered in isolation
before, the present paper is, to the best of our knowledge, the first to capture both elements in conjunction.
Indeed, {\it stability of wireless networks} has been widely studied in the
literature, see for instance~\cite{BF10,BFS09,RSS09a,TE92,WSP06}.
These papers restrict the attention though to {\it discrete topologies}, where the users can only reside in a finite number of locations. While a discrete network structure is a reasonable assumption in
case of a relatively small number of long-lived sources, it is less
suitable in case of a relatively large number of short-lived flows.
The latter scenario is increasingly relevant as emerging wireless
networks support traffic generated by massive numbers of nodes which
each individually may only engage in sporadic transmission activity.
The continuous spatial setting also provides useful insights in the
scaling behavior of discrete topologies as the number of nodes
grows large.

Stability of queueing networks in {\it continuous space} is investigated for instance in~\cite{AL94,LU10,Robert10}.
These papers prove stability of networks in which {\it only one user is allowed to transmit} at a time.
In contrast, the present paper focuses on the more complex situation
of simultaneous transmissions as governed by a scheduling discipline. 
The model in~\cite{BF11} can be interpreted as a wireless network on a continuum of locations in which the users are scheduled on a network-wide first-come-first-served basis. In this paper we study random admissible-set scheduling and show that it utilizes the spectrum in a more efficient manner.

From a methodological perspective, the continuous spatial setting
involves major additional challenges compared to a discrete network
structure.
Since users reside in a continuum of locations, the evolution of
the system cannot be represented in terms of a Markov chain with
a finite state space, and we therefore introduce a measure-valued
process as a description of the system.
In order to prove stability, we identify a Lyapunov function which
has a negative drift for all but a `small' set of states,
so that the Markov chain is positive Harris recurrent.

To use scheduling algorithms for discrete networks in a continuous setting, one can divide the space into a finite number of regions and treat every region as a node of a discrete network. In this setting a subset of nodes is, what we call, {\it guaranteed} if any subset of corresponding particles is admissible, regardless of the exact locations of the particles within each of the regions. One can then use a throughput-optimal scheduling discipline such as MaxWeight~\cite{TE92,VBY10} to select among the guaranteed regions. In this paper we will show that random admissible-set scheduling achieves stability whenever stability could be achieved by any region-based scheduling discipline, and thus random admissible-set scheduling is throughput-optimal whenever a region-based scheduling discipline is.

The remainder of the paper is organized as follows.
In Section~\ref{modeld} we show how the evolution of the system
may be described in terms of a Markov chain with a measure-valued
state space.
The main stability result is presented in Section~\ref{mainr},
along with an interpretation and discussion of its ramifications.
In Section~\ref{stability} we provide the proof of our main result,
and in particular identify a Lyapunov function which has negative
drift for all but a `small' set of states, and plays a critical
role in the proof.
In Appendix~\ref{prelim} we recall various useful definitions
and collect some preliminaries that are needed in order to apply
the Foster-Lyapunov approach for our specific Markov chain. Finally, Appendix~\ref{proofdetailsRAS} contains some proofs that have been relegated from the main text.
A preliminary version of this paper with partial results for the symmetric case appeared in~\cite{BBL11RAS}. 

\section{Model description} \label{modeld}

Consider a compact space~$\comps\subset \mathbb{R}^d$ for some $d<\infty$. Denote by $A(t,B)$ the number of particles arriving during the $t$-th time slot in $B \subseteq \comps$. Particles arrive in batches, and the batch size has a general non-negative discrete distribution with mean $\beta$, independent of the sizes and locations of other batches. The number of batches that arrive during a time slot has a general non-negative discrete distribution with mean $\lambda$ and is independent of the number of batches that arrive in other time slots.
That is, the numbers $A(t,\comps)$, $t=1,2,\dots$, are i.i.d.~copies of a non-negative random variable~$A$ with $\expect{A} = \lambda \beta$. Further we assume that $\expectsmall{A\log (A)|A>0}<\infty$ and that $0<\pr{A>0}<1$. The locations of the arriving batches are independent and distributed according to some measure $\nu(\cdot)$ on $H$, with $\nu(H) = 1$. Thus the expected number of particles to arrive to a subspace $B \subseteq \comps$ in one time slot is given by $\expect{A} \nu(B)$. We denote the number of particles in the space~$\comps$ at the start of the $t$-th time slot by~$Y(t)$, with $Y(t)=(Y(t,B),B\subseteq \comps)$ and $Y(t,B)$ denoting the number of particles residing in the subspace~$B$ at the start of the $t$-th time slot.
The state space of this process is denoted by $\Psi$ and consists
of all finite counting measures on~$\comps$. So, when
$y\in\Psi$, $y(B)$ denotes the number of particles residing in~$B$. In particular, $y(\{x\})$ denotes the number of particles at $x \in \comps$.

\paragraph{Random admissible-set scheduling.}

At the start of every time slot an admissible subset of particles will be removed. Here, $z \in \Psi$ is called a subset of $y \in \Psi$ if $z(\{x\}) \leq y(\{x\})$ for all $x \in \comps$.
To decide whether a set is admissible we define a function $F: \Psi \to \{0,1\}$ and we call $y\in \Psi$ admissible if and only if $F(y)=1$.
The function $F$ follows from the characteristics of the application, e.g.~for wireless networks the function~$F(\cdot)$ follows from the interference constraints, so that an admissible set consists of users whose transmissions are all successful if only the users in that set are transmitting at the same time. We will assume that $F(y)=1$ only if $y(\{x\})\in\{0,1\}$ for all $x \in \comps$ and that $F(y)=1$ whenever $y(\comps)\leq 1$. That is, we assume that at most one of the particles located at any location $x\in \comps$ can be removed in a single time slot and that sets consisting of at most one particle are always admissible.

Let $\chi (y)$ be the set of all subsets of $y$, i.e.
\[
\chi (y) = \{z \in \Psi : z(\{x\}) \leq y(\{x\}), \forall x \in \comps\},
\]
and let $R(t,Y(t),B)$ be the number of particles removed from $B\subseteq \comps$ in the $t$-th time slot, given the configuration $Y(t)$. An admissible subset of particles is selected uniformly at random. Hence, given $Y(t)=y$, $R(t,y)=z$ with probability
\[
\frac{F(z) \prod\limits_{\substack{x\in \comps,z(\{x\})>0}} y(\{x\})}{\sum_{u\in\chi(y)} F(u) \prod\limits_{\substack{x\in \comps,u(\{x\})>0}} y(\{x\})},
\]
where $R(t,y)=(R(t,y,B),B \subseteq \comps)$. Note that for $z\not\in\chi(y)$ we have $F(z)=0$ or there exists an $x \in \comps$ such that $z(\{x\})=1$ and $y(\{x\})=0$, so this probability is always zero in this case. Here and in the remainder of this paper, the value of the empty product is defined as $1$.

\paragraph{System dynamics.}

The evolution of $Y(t,B)$ is then described by the recursion
\[
Y(t,B) = Y(t-1,B) + A(t-1,B) - R(t,Y(t^{-}),B).
\]
Further, $R(t,Y(t^{-}),B)$ depends on the number of particles just before the start of the $t$-th time slot, 
\[
Y(t^{-},B) = Y(t-1,B) + A(t-1,B).
\]
From this description it follows that $(Y(t))_{t\in \mathbb{N}_0}$ is a Markov chain. We will equip the state space~$\Psi$ of the Markov chain with the smallest $\sigma$-field $\mathcal{B}(\Psi)$ with respect to
which the map $y \rightarrow y(B)$ is measurable for any Borel set
$B \subseteq \comps$. That is, we equip $\Psi$ with the Borel
$\sigma$-field as we will prove in Lemma~\ref{countgen}.

\paragraph{Guaranteed sets.}

Let $\mathcal{P}=\{P_1,\dots,P_K\}$ be a partition of the space~$\comps$, i.e.~$\bigcup_{i=1}^K P_i = \comps$, $P_i \neq \emptyset$ for all $i$ and $P_i \cap P_j = \emptyset$ for all $i \neq j$.
Further, let $\Theta_{\mathcal{P}}=2^{\mathcal{K}}$ be the collection of all subsets of $\mathcal{K}=\{1,\dots,K\}$. We call a set $S \in \Theta_{\mathcal{P}}$ guaranteed if any subset of particles, with exactly one residing in each of the regions contained in~$\bigcup_{k \in S} \{P_k\}$, is admissible, regardless of the exact locations within each of the regions. We denote by $\Omega_{\mathcal{P}} \subseteq \Theta_{\mathcal{P}}$ the collection of all guaranteed sets for the partition $\mathcal{P}$. Finally, we denote by $A_k$ the number of arrivals in the $k$-th region, so that $\expect{A_k}=\lambda \beta \nu(P_k)$.

\section{Main result} \label{mainr}

The following theorem states the main result of this paper and is a shorter version of Theorem~\ref{mainthm}, which is proven at the end of Section~\ref{stability}.
\begin{thm} \label{mainresult}
Assume that there exists a partitioning $\mathcal{P}$ of the space $\comps$, with $\nu(P_k)>0$ for $k=1,\dots,K$, such that there exist constants $\alpha_S \geq 0$, for which $\sum_ {S\in \Omega_{\mathcal{P}}} \alpha_S = 1$ and $\expect{A_k} < \sum_ {S\in \Omega_{\mathcal{P}}} \alpha_S {\rm I}_{\{k \in S \}}$ for all $k=1,\dots,K<\infty$. Then, the Markov chain $(Y(t))_{t\in \mathbb{N}_0}$ is positive Harris recurrent.
\end{thm}
In words, Theorem~\ref{mainresult} states that random admissible-set scheduling achieves stability whenever there exists a partition of the space such that the vector of arrival rates in each region lies in the interior of the convex hull of the incidence vectors of the guaranteed sets. The same stability condition holds for some scheduling algorithms that select among the guaranteed sets and then arbitrarily remove one packet in each of the regions contained in the selected guaranteed set, and is in fact known to be necessary in that case. For example, considering the regions as the queues of the network and the total number of packets residing in a region as its queue size, the MaxWeight algorithm~\cite{TE92} can be used to select among the guaranteed sets and gives the same stability condition. 

Thus random admissible-set scheduling is superior to region-based scheduling, in the sense that it always achieves stability whenever that could be achieved by any region-based scheduling discipline. Also note that a region-based scheduling discipline would require explicit knowledge of the partition that provides stability for its implementation, whereas for random admissible-set scheduling the partition is only required to formulate the sufficient stability condition in Theorem~\ref{mainresult}. 

As region-based MaxWeight is throughput-optimal for a number of scenarios~\cite{VBY10}, we can conclude that random admissible-set scheduling is throughput-optimal for these scenarios as well. Furthermore, under suitable smoothness conditions it seems plausible that there always exists a partition of the space such that stability can be achieved using a region-based scheduling discipline, whenever that is feasible to do so at all, so that random admissible-set scheduling is throughput-optimal under these conditions. However, characterizing the general smoothness conditions and proving this claim is a nontrivial analytical problem, and is beyond the scope of this paper.

These throughput-optimality guarantees might at first seem somewhat surprising because random admissible-set scheduling does not explicitly consider the size, or the congestion level, of the selected sets. However, suppose that the number of the particles in all areas but one is fixed. Then an arbitrary admissible set becomes more likely to contain a particle from the area with a non-fixed number of particles as that area becomes more densely populated. Thus, the selected set with random admissible-set scheduling indirectly does depend on the spatial congestion. Furthermore, note that, in this sense, random admissible-set scheduling resembles the queue-based CSMA algorithm, which is, similar to the MaxWeight algorithm, throughput-optimal in discrete topologies~\cite{RSS09a}. Although these observations do not immediately explain why random admissible-set has the same stability guarantees, they provide a starting point for constructing the Lyapunov function we use in Section~\ref{stability} to prove Theorem~\ref{mainresult}. 

It is worth emphasizing that we consider a particle-based version of random admissible-set scheduling rather than a node-based incarnation, in the sense that the strategy selects among sets of particles rather than sets of nodes.
While this distinction is immaterial when the measure $\nu(\cdot)$ is an absolutely continuous density and all batches have size one, the issue does become relevant when the batches can have a size greater than one or the measure $\nu(\cdot)$ has mass in discrete points.
When the measure $\nu(\cdot)$ has mass in discrete points, it may readily be concluded that a node-based version of random admissible-set scheduling is not throughput-optimal, as the empty set will be scheduled in any time slot with non-vanishing probability. 
Interestingly, this observation contrasts with the fact that the node-based version of the MaxWeight scheduling strategy guarantees throughput-optimality when the measure $\nu(\cdot)$ is a purely discrete distribution, whereas a particle-based version may fail to do so~\cite{VBY10}.

Under certain symmetry assumptions it is easier to show that random admissible-set scheduling is throughput-optimal, as we will show next. Also we can state the stability condition explicitly in this case. Let $\mu=\max \{y(\comps):F(y)=1\}$ denote the maximum number of particles in an admissible set. An immediate consequence of Theorem~\ref{mainresult} is the following corollary.
\begin{cor} \label{cormainresult}
Assume that there exists a partitioning $\mathcal{P}$ of the space $\comps$ such that $K \bmod{\mu}=0$, $\nu(P_k)=1/K$, for all $k=1,\dots,K$, and  $F(y)=1$ for any $y\in\Psi$ with $y(P_{m+j})=1$, for all $j=1,\dots,\mu$, with $m<K$ such that $m \bmod{\mu}=0$. Then, the Markov chain $(Y(t))_{t\in \mathbb{N}_0}$ is positive Harris recurrent if $\lambda \beta < \mu$.
\end{cor}
\begin{proof}
Choose $\alpha_S=\frac{\mu}{K}$ when $S=\{m+1,\dots,m+\mu\}$, with $m<K$ such that $m \bmod{\mu}=0$, and $\alpha_S=0$ otherwise. Noting that sets of the form $\{m+1,\dots,m+\mu\}$, with $m<K$ such that $m \bmod{\mu}=0$, are guaranteed by assumption, the stated result follows from Theorem~\ref{mainresult}.
\end{proof}
The assumptions of Corollary~\ref{cormainresult} ensure that the space~$\comps$ can be partitioned using a finite number of equally-sized disjoint sets such that taking any combination of elements of maximum size ($\mu$) from certain sets will always give an admissible set. As these sets all contain particles with nonzero probability, this assumption ensures that there exists an admissible set of size $\mu$ in any time slot with nonzero probability. Furthermore, there are assumed to be $K/\mu$ disjoint sets with this property, i.e.~any particle is in at least one guaranteed set of maximum size. 

These assumptions are for example satisfied when the arrival density is uniform, i.e.~$\nu(\cdot)$ is the Lebesgue measure, and the function $F(\cdot)$ and space $\comps$ are such that $F(\cdot)$ is translation-invariant. An example of this is given below.

\begin{exmp} \label{ucircle}
Let $\comps$ be the unit circle, where we use the interval $[0,1)$ to denote points on the circle, and let $\nu(\cdot)$ be the Lebesgue measure. Consider the protocol model for interference in wireless networks~\cite{GK00}, so that particles can be removed simultaneously whenever the distance between these particles is at least $r$, i.e.~$F(y)=1$ if $y(\{x\})\in\{0,1\}$ for all $x \in [0,1)$ and $y(\{x\})y(\{w\})>0$ only if $D(x,w)\geq r$ for all $x\neq w \in [0,1)$, with $D(x,w)=\min(|x-w|,1-|x-w|)$. Take $K \geq 2 \mu / (1 - \mu r)$ and $K$ a multiple of $\mu=\lfloor 1/r \rfloor$. Further, for $i = 1, \dots, K$, take
\[
P_i=\left[ \frac{\lfloor \frac{i-1}{\mu} \rfloor + \frac{K}{\mu}((i-1) \bmod{\mu})}{K} , \frac{1+\lfloor \frac{i-1}{\mu} \rfloor + \frac{K}{\mu}((i-1) \bmod{\mu})}{K} \right).
\]
We then see that the assumptions of Corollary~\ref{cormainresult} are satisfied whenever~$1/r$ is non-integer. 

Note that the probability that a set of size~$1/r$ is removed is almost surely zero in any time slot if~$1/r$ is integer-valued. In that case we therefore adjust $F(\cdot)$ slightly and take $F(\cdot)$ as above, but with $F(y)=0$ whenever $y(\comps)=1/r$, so that $\mu=\lceil 1/r \rceil-1$. We then see the assumptions of Corollary~\ref{cormainresult} are satisfied with $K$ and the sets $P_i$ as defined above.

We can thus conclude that, for this model, random admissible-set scheduling achieves stability if $\lambda \beta < \lceil 1/r \rceil-1$. Hence, as $\lceil 1/r \rceil-1$ is the maximum number of packets that can be transmitted simultaneously, is follows that random admissible-set scheduling is throughput-optimal.
\end{exmp}

The result of Corollary~\ref{cormainresult} may be interpreted as follows.
Suppose that the total number of particles in the system is large.
Then there will be a large number of admissible sets of size~$\mu$,
assuming that the particles are sufficiently dispersed across the
network and not concentrated in a few dense areas.
In fact, the number of admissible sets of size~$\mu$ will be
overwhelmingly large compared to the number of admissible sets of
smaller size.
By virtue of random admissible-set scheduling, one of the
admissible sets of size~$\mu$ will then be selected with high probability.
Thus, the expected number of removed particles will exceed the
expected number of arriving particles, provided $\lambda \beta < \mu$,
implying a reduction in the expected number of particles in the system,
and preventing the number of particles from growing without bound.

While the spatial dispersion of particles under random admissible-set
scheduling is intuitively plausible under the assumptions of Corollary~\ref{cormainresult}, it is certainly not obvious.
This is perhaps best illustrated through the following example, that shows that the result of
Corollary~\ref{cormainresult} may not necessarily hold for seemingly
similar but subtly different scheduling disciplines and that the particles in the system are almost all concentrated in one dense area.

\begin{exmp} \label{maxwithprior}
Consider the situation of Example~\ref{ucircle}. Instead of selecting an admissible set at random, now consider the scheduling discipline that gives priority to particles that are closest, in anticlockwise direction, to a certain point~$\zeta$ on the circle. Obeying this priority rule, we select as many particles as possible to remove in a certain time slot. That is, the particle closest to the given point gets removed, the particle closest to the point and at least a distance~$r$ away from the first particle gets removed, and so on until no particle can be selected anymore. We call this service discipline maximal scheduling with priorities.

\begin{figure}
  \centering
  \input{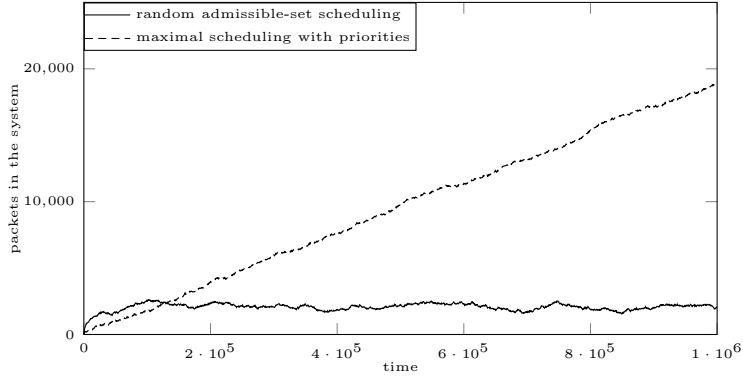}
  \caption{Random admissible-set scheduling and maximal scheduling with priorities for $\zeta=0.5$, $r=0.49$ and $\lambda=1.95$.}
  \label{stabcircle}
\end{figure}

Figure~\ref{stabcircle} shows a simulation result for both scheduling
disciplines with $\zeta=0.5$, $r=0.49$ and $\lambda=1.95$ starting from an empty
configuration and running for $10^6$ time slots. That is, the figure gives a realization of $Y(t,\comps)$ given that $Y(0,\comps)=0$ for $t=1,\dots,10^6$.
We see that at the start of the simulation the number of
particles in the system with random admissible-set scheduling grows faster than
the number of particles in the system with maximal scheduling with priorities.
This is because maximal scheduling with priorities always selects a subset
of maximal size obeying the priority rules, whereas random admissible-set
scheduling always selects admissible sets of a small size with a certain
probability, which gets lower as the number of particles in the network grows.
More importantly, after some time the number of particles in the system with
random admissible-set scheduling settles around an equilibrium value whereas the number of particles in the system with maximal scheduling with priorities keeps on growing linearly. This suggests that maximal scheduling with priorities is not stable while random admissible-set scheduling is stable for the chosen parameters. The latter will be proven in the next section.

\begin{figure}
  \centering
  \input{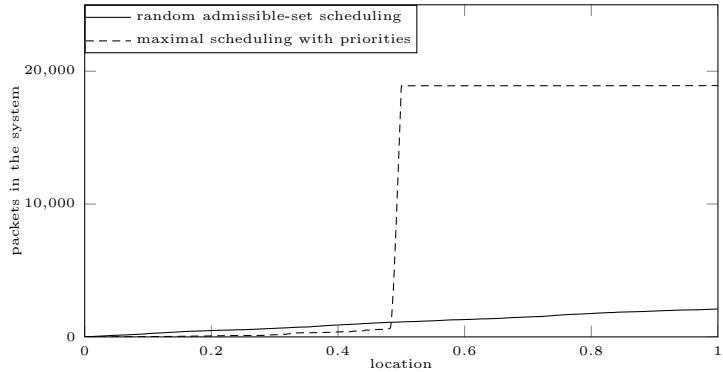}
  \caption{Terminal configuration of the simulation of Figure~\ref{stabcircle}.}
  \label{termconf}
\end{figure}

Figure~\ref{termconf} shows the terminal configuration of the
simulation of Figure~\ref{stabcircle}, i.e.~it gives a realization
of $Y(10^6,B)$ given that $Y(0,1)=0$ for $B=[0,s)$, $0<s\leq 1$, for both
scheduling disciplines.
For random admissible-set scheduling we see that the number of
particles in the interval $[0,s)$ is roughly linear in~$s$,
indicating that the particles are evenly spread out over the circle.
For maximal scheduling with priorities we observe that the number
of particles in $[0,s)$ slowly increases with~$s$ up to approximately
$s=0.48$, after which the number of particles in the system steeply
rises up to $s=0.5$.
For $s\geq 0.5=\zeta$ the number of particles in the interval $[0,s)$
is (almost) constant, implying that virtually no particles are
located in the interval $[0.5,1)$. Note that particles in the interval $[0.48,0.5)$ have the lowest priority and hence are, whenever they are allowed to, almost always removed simultaneously with other particles, as there are quite some particles in the system and outside this interval. However, the particles that are allowed to be removed simultaneously with particles in $[0.48,0.5)$ are also allowed to be removed simultaneously with some particles in $[0.5,0.52)$, which have the highest priority.
So we infer that the particles in this system are clustered in
$[0.48,0.5)$, and that too high a fraction of the time (larger then 0.02$\lambda$) no particle in this interval is removed, making the system unstable.
\end{exmp}

For the more general case considered in Theorem~\ref{mainresult} not all particles are necessarily in a guaranteed set of maximum size, which might result in an increase of the expected number of particles in the system when such particles are selected. Also, the particles might not be spread out, but be clustered in a few hot spots. Thus the above heuristic explanation of Corollary~\ref{cormainresult} does not directly extend to the more general case in Theorem~\ref{mainresult}. Nevertheless, random admissible-set scheduling guarantees stability when the condition in Theorem~\ref{mainresult} is fulfilled as we prove in the next section. 

\section{Proof} \label{stability}

The proof of Theorem~\ref{mainresult} relies on the Foster-Lyapunov
criteria and involves the identification of a function which has
negative drift for all but a small set of states.
Appendix~\ref{prelim} contains several useful definitions
and preliminaries that are needed to apply the Foster-Lyapunov
approach for our specific Markov chain.

We will from now on suppose that the assumptions of Theorem~\ref{mainresult} hold and we consider a partitioning $\mathcal{P}=\{P_1,\dots,P_K\}$ of the space~$\comps$, with $\nu(P_k)>0$ for $k=1,\dots,K$, with associated constants $\alpha_S \geq 0$, for which $\sum_ {S\in \Omega_{\mathcal{P}}} \alpha_S = 1$ and $\expect{A_k} < \sum_ {S\in \Omega_{\mathcal{P}}} \alpha_S {\rm I}_{\{k \in S \}}$ for all $k=1,\dots,K<\infty$. Further, we take $\epsilon>0$ such that $\expect{A_k} \leq (1-2\epsilon) \sum_{S \in \Omega_{\mathcal{P}}} \alpha_S {\rm I}_{\{k \in S\}}$.
To simplify the notation, we will use the notation $\Theta$ for all subsets of $\{1,\dots,K\}$ and $\Omega$ for all guaranteed sets for the partition~$\mathcal{P}$.

Let $x_k(y)$ be the number of particles residing in the $k$-th region, $P_k$, given configuration $y \in \Psi$. We will denote by $\Theta(y)$ the subsets containing particles, i.e.~$\Theta(y)=\{S\in\Theta :x_k(y)\geq 1,\forall k \in S\}$, and by $\Omega(y)$ the guaranteed sets containing particles,~$\Omega(y)=\{S\in\Omega :x_k(y)\geq 1,\forall k \in S\}$. Further, denote by $q_S(y)$ the probability that the particles that get removed belong to the subset of regions $S \in \Theta$ given that the system is in state~$y$, with $p_k(y) = \sum_{S \in \Theta: S \ni k} q_S(y)$ the probability that a particle in the $k$-th region gets removed given configuration $y \in \Psi$.

For $S \in \Theta$, denote
\[
w_S(y) = \prod\limits_{k \in S: x_k(y) \geq 1} x_k(y)
\]
and define
\[
w(y) = \max\limits_{S \in \Omega (y)} w_S(y).
\]
Further define for any $\delta>0$
\[
B(\delta)=\left\{y\in\Psi : w(y)\geq\left(\frac{2|\Theta|}{\delta}\right)^{2/\delta}\right\}.
\]

In the next lemma we will show that the expected value of $\sum_{k \in S} \log(x_k)$, for the set of particles~$S$ selected with random admissible-set scheduling, is at least close to the maximum possible value over all guaranteed sets for all states $y\in B(\delta)$, which is in fact the guaranteed set that would be selected if the weighted MaxWeight algorithm~\cite{ESP05,SW06} would be used to select among the guaranteed sets.

\begin{lem} \label{boundedw}
For all states $y\in B(\delta)$ we have,
\[
\sum_{S \in \Theta(y)} q_S(y) \log(w_S(y)) \geq (1 - \delta) \log(w(y)).
\]
\end{lem}
\begin{proof}
The proof proceeds along similar lines as in~\cite{BF10,NTS10}.

Define
\[
\Upsilon(y) =
\{S \in \Theta(y): \log(w_S(y)) \geq (1 - \frac{\delta}{2}) \log(w(y))\}.
\]
Then,
\beeq \label{lb1}
\sum_{S \in \Theta(y)} q_S(y) \log(w_S(y)) \geq
(1 - \frac{\delta}{2}) \log(w(y)) \sum_{S \in \Upsilon(y)} q_S(y).
\eneq
For any $S \in \Theta$, let $v_S(y)$ be the number of admissible
subsets of particles of size $|S|$ with exactly one residing in each of the
regions contained in~$S$, with the convention that
$v_{\emptyset}(y)=1$ for all $y \in \Psi$ as the empty set is always admissible.
Because every admissible subset can have
at most one particle residing in each region, there is exactly one
$S \in \Theta$ for which this subset is counted in $v_S(y)$. Thus
the total number of admissible subsets of particles is given by
$\sum_{T \in \Theta} v_T(y)$ and, as an admissible subset of
particles is selected uniformly at random,
\begin{equation}
\label{defqS}
q_S(y) = \frac{v_S(y)}{\sum_{T \in \Theta} v_T(y)}=\frac{v_S(y)}{\sum_{T \in \Theta(y)} v_T(y)}.
\end{equation}
Further observe that $v_S(y) \leq w_S(y)$ for all $S \in \Theta (y)$, with
equality for all $S \in \Omega(y)$.

Thus,
\begin{equation}
\label{qSinUpsbound}
\sum_{S \not\in \Upsilon(y)} q_S(y) =
\frac{\sum_{S \not\in \Upsilon(y)} v_S(y)}{\sum_{S \in \Theta(y)} v_S(y)} \leq
\frac{\sum_{S \not\in \Upsilon(y)} w_S(y)}
{\sum_{S \in \Omega(y)} w_S(y)} \leq
\frac{|\Theta| w(y)^{1 - \frac{\delta}{2}}}{w(y)} =
|\Theta| w(y)^{- \frac{\delta}{2}}.
\end{equation}
The latter quantity is less than $\frac{\delta}{2}$ for all
states~$y\in B(\delta)$, and thus
\beeq \label{lb2}
\sum_{S \in \Upsilon(y)} q_S(y) \geq 1 - \frac{\delta}{2}.
\eneq
Combining the lower bounds (\ref{lb1}) and (\ref{lb2}), we obtain
\[
\sum_{S \in \Theta(y)} q_S(y) \log(w_S(y)) \geq
(1 - \frac{\delta}{2}) \log(w(y)) (1 - \frac{\delta}{2}) \geq
(1 - \delta) \log(w(y))
\]
for all states~$y\in B(\delta)$.
\end{proof}

Define the nonnegative function $V: \Psi \to {\mathbb R}$ by
\[
V(y) = \sum_{k: x_k(y) \geq 1} x_k(y) \log(x_k(y)).
\]
and the function $G: \Psi \to {\mathbb R}$ by
\[
G(y) = \sum_{k: x_k(y) \geq 1} \log(x_k(y)) \Big[\expect{A_k} - p_k(y)\Big].
\]
Note that
\[
V(y) = \sum_{k=1}^{K} x_k(y) \log(\max(x_k(y),1))
\]
and
\[
V(y) = \sum_{k: x_k(y) \geq 2} x_k(y) \log(x_k(y)).
\]
Further observe that the function $V(\cdot)$ only depends on~$y$
through the values of $x_k(y)$.
However, the $x_k(y)$'s do not constitute a Markov chain, and hence
we need to treat $V(\cdot)$ as a function of the full state
description~$y$ in order for the Foster-Lyapunov approach to apply.

We now first find the relation between the drift of $V(y)$ and $G(y)$, where the drift of $V(y)$ is defined by
\[
\Delta V(y) = \expect{V(Y(t + 1)) | Y(t) = y} - V(y).
\]
The proof of this lemma is deferred to Appendix~\ref{proofdetailsRAS}.

\begin{lem} \label{bounded}
$\Delta V(y)=G(y)+G_2(y)$, with $G_2(y)$ a bounded function.
\end{lem}

We proceed by finding an upper bound on $G(y)$ for all states $y \in B(\epsilon)$.

\begin{lem} \label{upG}
For all states $y \in B(\epsilon)$,
\[
G(y) \leq - \epsilon \sum_{k: x_k(y) \geq 1} \expect{A_k} \log(x_k(y)) .
\]
\end{lem}
\begin{proof}
We have by assumption
\begin{equation}
\label{firstboundG}
G(y) \leq \sum_{k: x_k(y) \geq 1} \log(x_k(y)) \Big[(1-2\epsilon) \sum_{S \in \Omega} \alpha_S {\rm I}_{\{k \in S\}} - p_k(y)\Big].
\end{equation}
Now note that, for all states $y \in B(\epsilon)$, we may write
\bearno
\sum_{k: x_k(y) \geq 1} \log(x_k(y)) p_k(y)
&=&\sum_{k: x_k(y) \geq 1} \log(x_k(y)) \sum_{S \in \Theta: S \ni k} q_S(y) \\
&=&\sum_{S \in \Theta} q_S(y) \sum_{k \in S: x_k(y) \geq 1} \log(x_k(y)) \\
&=&\sum_{S \in \Theta(y)} q_S(y) \log(w_S(y)).
\enarno
Likewise, we may write
\bearno
\sum_{k: x_k(y) \geq 1} \log(x_k(y)) \sum_{S \in \Omega} \alpha_S {\rm I}_{\{k \in S\}}
&=& \sum_{S \in \Omega} \alpha_S \sum_{k \in S: x_k(y) \geq 1} \log(x_k(y)) \\
&=& \sum_{S \in \Omega(y)} \alpha_S \log(w_S(y)).
\enarno
Substitution of these two equalities in~\eqref{firstboundG} gives
\bearno
G(y)
&=&- \epsilon \sum_{k: x_k(y) \geq 1} \log(x_k(y)) \sum_{S \in \Omega} \alpha_S {\rm I}_{\{k \in S\}} \\
&+& (1 - \epsilon) \sum_{S \in \Omega(y)} \alpha_S \log(w_{S}(y)) - \sum_{S \in \Theta(y)} q_S(y) \log(w_S(y)).
\enarno
Then, using Lemma~\ref{boundedw} and noting that $w_{S}(y) \leq w(y)$, yields
\bearno
G(y)
&\leq&
- \epsilon \sum_{k: x_k(y) \geq 1} \log(x_k(y)) \expect{A_k} + (1 - \epsilon)
\Big[\sum_{S \in \Omega(y)} \alpha_S \log(w_{S}(y)) - \log(w(y))\Big] \\
&\leq&
- \epsilon \sum_{k: x_k(y) \geq 1} \log(x_k(y)) \expect{A_k},
\enarno
for all states~$y\in B(\epsilon)$.
\end{proof}

By Lemma~\ref{bounded} we know that $\Delta V(y)\leq G(y)+G_2^{\max}$, where $G_2^{\max} = \sup_{y \in \Psi} G_2(y) < \infty$. Now consider the set $C$ where the drift
of $V(y)$ might be bigger than $-1$, i.e.~consider
\[
C=\{y\in\Psi : G(y) \geq -G_2^{\max} - 1\}.
\]
The next lemma, whose proof can be found in Appendix~\ref{proofdetailsRAS}, shows that this set is {\it small} (see Definition~\ref{defsmall}).

\begin{lem} \label{small}
The set $C$ is small.
\end{lem}

Having established these lemmas, we are now in the position to prove stability of our system.
\begin{thm} \label{mainthm}
The Markov chain $(Y(t))_{t\in \mathbb{N}_0}$
is positive Harris recurrent with invariant probability measure $\pi$ and
\[
\pi(f)=\int \pi (dx)f(x)<\infty,
\]
where
\[
f(y) =
\left\{\begin{array}{ll}
-G(y)-G_2^{\max}, & y \not\in C,\\
1 & y \in C.
\end{array}\right.
\]
Moreover,
\[
\lim_{t\rightarrow \infty} \expect{g(Y(t))|Y(0)=y} = \int \pi (dx)g(x),\quad \forall y \in \Psi,
\]
for any function~$g$ satisfying $|g(x)|\leq c (f(x)+1)$ for all~$x$ and some $c<\infty$.
\end{thm}
\begin{proof}
In Lemma~\ref{irapsm} we have proven that our Markov chain
satisfies the irreducibility and aperiodicity properties
of Theorem~\ref{thmMT}. Further, we have proven in Lemma~\ref{small}
that the set $C$ is small. Thus, as $f\geq 1$ by construction
and $V$ is nonnegative and finite everywhere, we need to show that
\beeq \label{drift2}
\Delta V(y)\leq -f(y) + b{\rm I}_C(y),\quad \forall y \in \Psi,
\eneq
for some constant $b\in \mathbb{R}$, in order to prove our claim.

For $y \not\in C$ we get
\[
\Delta V(y)\leq G(y)+G_2^{\max},
\]
which holds as we have shown in Lemma~\ref{bounded}. Further, for $y \in C$ we get,
\[
\Delta V(y)\leq -1+b.
\]
We therefore take $b\geq 1+\sup_{y\in C} \Delta V(y)$, so
that the inequality holds by construction. Hence we have shown that~\eqref{drift2} holds for all $y\in \Psi$,
proving our claim.
\end{proof}

\begin{rem}
Lemma~\ref{irapsm} shows that our Markov chain to be $\varphi$-irreducible,
where $\varphi$ is the Dirac measure on $\Psi$ assigning unit mass to the
empty configuration $y_0$. Hence, by Definition~\ref{psiirr},
$\psi(\{y_0\})>0$. Theorem~\ref{mainresult} then follows by definition
of a Harris recurrent chain, see Definition~\ref{phc}.
\end{rem}

\begin{rem}\label{remarkRAS}
It is not strictly necessary to select uniformly at random among the admissible subsets for the proof arguments to hold. Consider for example `weighted' random admissible-set scheduling that gives weight $f(S,y)$ to the region $S$, so that
\[
q_S(y) =\frac{f(S,y) v_S(y)}{f(S,y) \sum_{T \in \Theta(y)} v_T(y)},
\]
with $f(S,y)\equiv 1$ for regular random admissible-set scheduling. It can readily be verified that all proof arguments hold for this scheduling discipline as well whenever $0<a\leq f(S,y) \leq b <\infty$, so that it has the same sufficient stability condition.
\end{rem}

\begin{rem}
The Foster-Lyapunov approach may also be leveraged to derive an
upper bound for the expected value of functions of the total number
of particles in the system. Specifically, define $J(y) = \sum_{k: x_k(y) \geq 1} \log(x_k(y))$,
and denote $G_1^{\max} = \sup_{y \in B(\epsilon)^c} (G(y)+J(y))<\infty$.
By virtue of Lemma~\ref{upG} we then have $G(y) \leq G_1^{\max} - \epsilon\gamma J(y)$, with $\gamma = \min_{k\in\mathcal{K}} \expect{A_k}>0$.
Using Lemma~\ref{bounded} and taking expectations yields
\[
\expect{V(Y(t + 1))} - \expect{V(Y(t))} \leq
- \epsilon\gamma \expect{J(Y(t))} + G_1^{\max} + G_2^{\max},
\]
for all $t = 1, 2, \dots$.

Summing over $t = 1, \dots, T$, we obtain
\[
\epsilon\gamma \sum_{t = 1}^{T} \expect{J(Y(t))} \leq
\expect{V(Y(1))} + T (G_1^{\max} + G_2^{\max}),
\]
and thus
\[
\limsup_{T \to \infty} \frac{\epsilon \gamma}{T}
\sum_{t = 1}^{T} \expect{J(Y(t))} \leq G_1^{\max} + G_2^{\max}.
\]
\end{rem}

\section{Concluding remarks} \label{conclusion}

We examined the stability of wireless networks whose users are distributed over a compact space. A subset of users is called admissible when their simultaneous activity obeys the prevailing interference constraints and, in each time slot, an admissible subset of users is selected uniformly at random to transmit one packet. We showed that this random admissible-set scheduling mechanism is superior to region-based scheduling, in the sense that it always achieves stability whenever that could be achieved by any region-based scheduling algorithm, and hence is throughput-optimal in a broad range of scenarios.

Although it is not required to divide the space into regions for random admissible-set scheduling, a partitioning of the space is used to formulate the general sufficient stability criterion. For symmetric scenarios we showed that the stability criterion can be formulated in an explicit manner, only requiring the system parameters. An interesting topic for further research is to formulate the general stability criterion in a more explicit manner as well.

With random admissible-set scheduling one has to find all admissible sets in each time slot, which can be a computationally complex task, especially when the number of particles in the system is very large and admissible sets can consist of many particles. For practical purposes it would thus be essential to find a distributed way to select one admissible set with a certain probability that fulfills the technical condition stated in Remark~\ref{remarkRAS}. A starting point in that direction might be the observation that random admissible-set scheduling resembles the queue-based CSMA algorithm.

\section{Acknowledgments}

This work was supported by Microsoft Research through its PhD Scholarship Programme, by the European Research Council (ERC) and by the Netherlands Organisation for Scientific Research (NWO).

\appendix
\section{Preliminary results} \label{prelim}

As mentioned earlier, the stability proof relies on
a Foster-Lyapunov approach.
In this appendix, we first recall various relevant definitions
and a result from Meyn and Tweedie~\cite{MT93}.
After that, we prove that the Markov chain $(Y(t))_{t\in \mathbb{N}_0}$
introduced in Section~\ref{modeld} satisfies all technical
conditions for the Foster-Lyapunov approach to apply.

Let $(\hat Y(t))_{t\in \mathbb{N}_0}$ be a Markov chain with state space
$\hat \Psi$. Further, let $\mathcal{B}(\hat \Psi)$ be the $\sigma$-field
of subsets of $\hat \Psi$. This $\sigma$-field is assumed to be countably
generated, i.e. it is generated by some countable class of subsets of~$\hat{\Psi}$.

\begin{defn}
$(\hat Y(t))_{t\in \mathbb{N}_0}$ is said to be $\varphi$-irreducible if there
exists a measure $\varphi$ on $\mathcal{B}(\hat \Psi)$ such that, whenever
$\varphi(C)>0$, we have
\[
\pr{\min(t:\hat Y(t) \in C)<\infty|\hat Y(0)=\hat y}>0,\quad \forall\hat y \in \hat \Psi.
\]
\end{defn}

Let $\prm{\hat y,C}$ denote the $m$-step transition probability to go from state
$\hat y$ to the set $C \in \mathcal{B}(\hat \Psi)$.
Further define the transition kernel
\[
K_{\frac{1}{2}}(\hat y,C)=\sum_{m=0}^{\infty} \prm{\hat y,C}2^{-(m+1)},\quad \hat y \in \hat \Psi, C \in \mathcal{B}(\hat \Psi).
\]

\begin{defn} \label{psiirr}
$(\hat Y(t))_{t\in \mathbb{N}_0}$ is said to be $\psi$-irreducible if it is
$\varphi$-irreducible for some $\varphi$ and the measure $\psi$ is a maximal
irreducibility measure, i.e. it satisfies the following conditions:
\begin{enumerate}[label=(\roman{*})]
\item For any other measure $\phi '$ the chain is $\phi '$-irreducible if
and only if $\psi(C)=0$ implies $\phi ' (C) = 0$.
\item If $\psi(C)=0$, then $\psi(\{\hat y : \pr{\min(t:\hat Y(t) \in C)<\infty|\hat Y(0)=\hat y}>0\})=0$.
\item The probability measure $\psi$ is equivalent to
\[
\psi ' (C) = \int_{\hat \Psi} \varphi ' (d \hat y) K_{\frac{1}{2}}(\hat y,C),
\]
for any finite measure $\varphi '$ such that the chain is $\phi '$-irreducible.
\end{enumerate}
\end{defn}
\begin{note}
By~\cite[Thm.~4.0.1]{MT93} we know that if there exists a measure $\varphi$
such that the chain is $\varphi$-irreducible, then there exists an (essentially
unique) maximal irreducibility measure $\psi$.
\end{note}

\begin{defn} \label{defsmall}
A set $C \in \mathcal{B}(\hat \Psi)$ is called $\xi_m$-small if there exists
an $m>0$ and a non-trivial measure $\xi_m$ on $\mathcal{B}(\hat \Psi)$,
such that
\[
\prm{\hat y,D}\geq \xi_m(D),\quad \forall \hat y \in C,D\in\mathcal{B}(\hat \Psi).
\]
A set is called small if it is $\xi_m$-small for some $m>0$ and some non-trivial
measure $\xi_m$.
\end{defn}

\begin{defn}
Suppose $(\hat Y(t))_{t\in \mathbb{N}_0}$ is $\varphi$-irreducible. The chain is
called strongly aperiodic when there exists a $\xi_1$-small set $C$ with
$\xi_1(C)>0$.
\end{defn}

Let ${\rm I}_C(x)$ be the indicator function of the set $C$, i.e.~${\rm I}_C(x)=1$
if $x \in C$ and $0$ otherwise.

\begin{defn} \label{phc}
A $\psi$-irreducible chain $(\hat Y(t))_{t\in \mathbb{N}_0}$ is said to be
Harris recurrent if for all $C \in \mathcal{B}(\hat \Psi)$ such that
$\psi(C)>0$ we have
\[
\pr{\sum_{t=1}^{\infty} {\rm I}_C(Y(t))=\infty|\hat Y (0) = \hat y}=1,\quad \forall \hat y \in C.
\]
If a Harris recurrent chain admits an invariant probability measure it is
called positive Harris recurrent.
\end{defn}

%
%
%

The following theorem follows from Chapter 14 in~\cite{MT93}.

\begin{thm} \label{thmMT}
Suppose that the chain $(\hat Y(t))_{t\in \mathbb{N}_0}$ is $\psi$-irreducible and
strongly aperiodic.
If there exists some small set $\hat C$, a function $\hat f \geq 1$ and some
nonnegative function $\hat V$ that is finite everywhere such that
\beeq \label{drift}
\Delta \hat V(\hat y) \leq -\hat f(\hat y) + \hat b {\rm I}_{\hat C}(\hat y),\quad \forall \hat y \in \hat \Psi,
\eneq
then $(\hat Y(t))_{t\in \mathbb{N}_0}$ is positive Harris recurrent with invariant
probability measure $\pi$ and
\[
\pi(\hat f)=\int \pi (dx)\hat f(x)<\infty.
\]
Moreover,
\[
\lim_{t\rightarrow \infty} \expect{\hat g(\hat Y(t))|\hat Y(0)=\hat y} =\int \pi (dx)\hat g(x),\quad \forall \hat y \in \hat \Psi,
\]
for any function~$\hat g$ satisfying $|\hat g(x)|\leq \hat c (\hat f(x)+1)$ for all~$x$ and some $\hat c<\infty$.
\end{thm}

To prove that our Markov chain fulfills the conditions in
Theorem~\ref{thmMT}, we first show that our $\sigma$-field, $\mathcal{B}(\Psi)$,
is countably generated.
\begin{lem} \label{countgen}
$\mathcal{B}(\Psi)$ is the Borel $\sigma$-field. Furthermore, $\mathcal{B}(\Psi)$ is
countably generated.
\end{lem}
\begin{proof}
In this proof we will use some definitions and results in measure theory,
see~\cite{DV03} and~\cite{Kal97} for more details.

First note that $\comps$ endowed with the topology generated by the open
sets defines a complete separable metric space as $\comps$ is compact. Then it follows by
~\cite[Thm.~A2.6.III]{DV03} that the Borel $\sigma$-field of $\Psi$ is
the smallest $\sigma$-field with respect to which the map
$y \rightarrow y(B)$ is measurable for any Borel set $B\subseteq \comps$.
Further it follows that $\Psi$ endowed with the {\it vague topology}
is a complete separable metric space. The vague topology is the topology
generated by the mappings $\phi \rightarrow \phi g = \int g d\phi$, with
$\phi$ a measure on $\Psi$, for all continuous functions
$g:\comps\rightarrow\mathbb{R}^{+}$ with compact support.

Since the space is separable, there exists a countable dense set $\mathcal{D}$
in this space. Let $\mathcal{S}_0$ be the class of all finite intersections of
all open sets $\{x\in \comps:D(x,d)<r\}$, with $d\in\mathcal{D}$ and $r\in \mathbb{Q}^+$. Then, by
~\cite[Lemma~A2.1.III]{DV03}, $\mathcal{S}_0$ is countable and generates
the Borel $\sigma$-field, $\mathcal{B}(\Psi)$.
\end{proof}

We will now prove that our Markov chain satisfies the irreducibility and
aperiodicity properties of Theorem~\ref{thmMT}.

\begin{lem} \label{irapsm}
$(Y(t))_{t\in \mathbb{N}_0}$ is $\varphi$-irreducible
and strongly aperiodic, where $\varphi$ is the Dirac measure on $\Psi$
assigning unit mass to the empty configuration $y_0$. Moreover, the level
sets of the form $L_m=\{y\in\Psi:y(\comps)\leq m\},m> 0,$ are small.
\end{lem}
\begin{proof}
The proof proceeds along similar lines as in~\cite{LU10}.

Consider an initial configuration $y$ with $y(\comps)=n\leq m$ particles, so
$y\in L_m$. Because $\pr{A=0}>0$ by assumption, the probability that the system is empty after $m$ time slots
is greater than the probability that no particles arrive during the first
$m$ time slots times the probability that the $n$ particles are served
in the first $m$ time slots. Thus, as in a non-empty configuration
the probability that at least one particle is served in a time slot is
at least $\frac{1}{2}$,
\[
\prm{y,\{y_0\}}\geq\pr{A=0}^m \left(\frac{1}{2}\right)^m,
\]
which is greater than zero as $\pr{A=0}>0$.
This proves that $(Y(t))_{t\in \mathbb{N}_0}$ is $\varphi$-irreducible in this case, because
$\varphi(D)>0$ only when $y_0 \in D$.

Now, define the measure $\xi_m=\pr{A=0}^m\left(\frac{1}{2}\right)^m\varphi$.
For this measure we have $\prm{y,D}\geq\xi_m(D)$ for all
$D\in\mathcal{B}(\Psi)$. So we see that $L_m$ is $\xi_m$-small.
Further, $\{y_0\}$ is $\xi_1$-small and $\xi_1(\{y_0\})>0$, thus
$(Y(t))_{t\in \mathbb{N}_0}$ is strongly aperiodic.

\end{proof}

\section{Proof details}\label{proofdetailsRAS}
In this section we provide the full proof details of Lemma~\ref{bounded} and Lemma~\ref{small}.

\subsection{Proof of Lemma~\ref{bounded}}
Remember that at most one particle can be removed from each region in every time slot. We thus get
\begin{align*}
\hspace{2em}&\hspace{-2em}\expect{V(Y(t + 1)) | Y(t) = y} \\
=&\expectbig{\sum_{k=1}^{K} x_k(Y(t + 1)) \log(\max(x_k(Y(t + 1)),1)) \Big | Y(t) = y} \\
=&\sum_{k=1}^{K}\expectbig{x_k(Y(t + 1)) \log(\max(x_k(Y(t + 1)),1)) \Big | Y(t) = y} \\
=&\sum_{k: x_k(y) \geq 2} p_k(y)\expect{(x_k(y)+A_k(t)-1) \log(x_k(y)+A_k(t)-1)} \\
&+\sum_{k: x_k(y) \geq 2} (1-p_k(y))\expect{(x_k(y)+A_k(t)) \log(x_k(y)+A_k(t))} \\
&+\sum_{k: x_k(y) = 1} p_k(y)\pr{A_k(t)\geq 1}\expect{(1+A_k(t)-1) \log(1+A_k(t)-1) | A_k(t)\geq 1} \\
&+\sum_{k: x_k(y) = 1} (1-p_k(y))\expect{(1+A_k(t)) \log(1+A_k(t))} \\
&+\sum_{k: x_k(y) = 0} \pr{A_k(t)\geq 1}\expect{A_k(t) \log(A_k(t)) | A_k(t)\geq 1}, \\
\end{align*}
We further have
\begin{align*}
\hspace{2em}&\hspace{-2em}\sum_{k: x_k(y) \geq 2} p_k(y)\expect{(x_k(y)+A_k(t)-1) \log(x_k(y)+A_k(t)-1)} \\
&+\sum_{k: x_k(y) \geq 2} (1-p_k(y))\expect{(x_k(y)+A_k(t)) \log(x_k(y)+A_k(t))} \\
=&\sum_{k: x_k(y) \geq 2} p_k(y)\expect{(x_k(y)+A_k(t)-1)\Big(\log(x_k(y)) + \log\Big(1+\frac{A_k(t)}{x_k(y)}-\frac{1}{x_k(y)}\Big)\Big)} \\
&+\sum_{k: x_k(y) \geq 2} (1-p_k(y))\expect{(x_k(y)+A_k(t))\Big(\log(x_k(y)) + \log\Big(1+\frac{A_k(t)}{x_k(y)}\Big)\Big)} \\
=&\sum_{k: x_k(y) \geq 2} (x_k(y)+\expect{A_k}-p_k(y))\log(x_k(y)) \\
&+\sum_{k: x_k(y) \geq 2} p_k(y)\expect{(x_k(y)+A_k(t)-1)\log\Big(1+\frac{A_k(t)}{x_k(y)}-\frac{1}{x_k(y)}\Big)} \\
&+\sum_{k: x_k(y) \geq 2} (1-p_k(y))\expect{(x_k(y)+A_k(t))\log\Big(1+\frac{A_k(t)}{x_k(y)}\Big)}. \\
\end{align*}
Now notice that for constants $a\geq 0,b\geq 0,c>0$
\begin{align*}
\hspace{2em}&\hspace{-2em}\expect{(a+A_k(t)) \log\Big(b+\frac{A_k(t)}{c}\Big)}\\
=&\expect{(a+A_k(t)) (\log (bc+A_k(t)) - \log (c))}\\
\leq&\expect{(a+A_k(t)) (\log (bc+A_k(t))-\log (c)) | A_k(t)\geq 1 )}\\
\leq&\expect{(a+A_k(t)) (bc+\log (A_k(t))-\log (c)) | A_k(t)\geq 1 )}\\
=& a(bc-\log (c)) + \expect{ bcA_k(t) + a \log (A_k(t)) | A_k(t)\geq 1 )}+ \expect{A_k(t) \log (A_k(t))| A_k(t)\geq 1 } \\
\leq& a(bc-\log (c)) + \frac{abc\expect{A_k(t)}}{1-\pr{A_k(t)=0}}+ \expect{A_k(t) \log (A_k(t))| A_k(t)\geq 1 }.
\end{align*}
Thus, as $\expect{A_k(t)}$ and $\expect{A_k(t) \log (A_k(t))| A_k(t)\geq 1 }$ are bounded, we find
\bear \label{boundedlog}
\expect{(a+A_k(t)) \log\Big(b+\frac{A_k(t)}{c}\Big)}<\infty.
\enar
Hence $\Delta V(y)=G(y)+G_2(y)$, with
\begin{align*}
G_2(y)=&\sum_{k: x_k(y) = 1} p_k(y)\pr{A_k(t)\geq 1}\expect{(1+A_k(t)-1) \log(1+A_k(t)-1) | A_k(t)\geq 1} \\
&+\sum_{k: x_k(y) = 1} (1-p_k(y))\expect{(1+A_k(t)) \log(1+A_k(t))} \\
&+\sum_{k: x_k(y) = 0} \pr{A_k(t)\geq 1}\expect{A_k(t) \log(A_k(t)) | A_k(t)\geq 1} \\
&+\sum_{k: x_k(y) \geq 2} p_k(y)\expect{(x_k(y)+A_k(t)-1)\log\Big(1+\frac{A_k(t)}{x_k(y)}-\frac{1}{x_k(y)}\Big)} \\
&+\sum_{k: x_k(y) \geq 2} (1-p_k(y))\expect{(x_k(y)+A_k(t))\log\Big(1+\frac{A_k(t)}{x_k(y)}\Big)}, \\
\end{align*}
which is a bounded function by~(\ref{boundedlog}) and as $K<\infty$.
\qed

\subsection{Proof of Lemma~\ref{small}}
Consider the sets
\[
\hat B(\epsilon)=\Big\{y\in\Psi:x_k(y)\leq\left(\frac{2|\Theta|}{\epsilon}\right)^{2/\epsilon}, \forall k\in\mathcal{K} \Big\},
\]
and
\[
\hat C=\Big\{y\in\Psi:\log x_k(y)\leq\frac{G_2^{\max}+1}{\epsilon \gamma}, \forall k\in\mathcal{K} \Big\},
\]
with $\gamma = \min_{k\in\mathcal{K}} \expect{A_k} = \min_{k\in\mathcal{K}} \lambda \beta \nu(P_k)>0$.

We see that $B(\epsilon)^c \subseteq \hat B(\epsilon)$ as subsets of one
region are guaranteed and thus $w(y)\geq x_k(y)$, for all
$k=1,\dots,K$. Further we see that
$C \setminus B(\epsilon)^c\subseteq \hat C$,
as follows from the upper bound for $G(y)$ found in
Lemma~\ref{upG}.
Hence,
\[
C \subseteq B(\epsilon)^c \cup (C \setminus B(\epsilon)^c)
\subseteq \hat B(\epsilon) \cup \hat C = \left\{y\in\Psi: x_k(y)\leq M, \forall k\in\mathcal{K}\right\},
\]
where
\[
M=\max\Big(
{\rm e}^{\frac{G_2^{\max}+1}{\epsilon\gamma}},\Big(\frac{2|\Theta|}{\epsilon}\Big)^{2/\epsilon}\Big).
\]
Thus $C \subseteq L_{KM}$, where $L_m$ is the level set,
$L_m=\{y\in\Psi:y(\comps)\leq m\}$. We know by Lemma~\ref{irapsm}
that $L_{KM}$ is small and hence, by definition, $C$ is small.
\qed

\bibliographystyle{abbrv}
\bibliography{RAS}

\end{document}